\documentclass[11pt]{amsart}
\usepackage{amsmath,amssymb,latexsym,soul, tikz}
\usepackage{cite,mathrsfs,amsfonts}
\usepackage{color,enumitem,graphicx}
\usepackage[colorlinks=true,urlcolor=blue,
citecolor=red,linkcolor=blue,linktocpage,pdfpagelabels,
bookmarksnumbered,bookmarksopen]{hyperref}
\usepackage[english]{babel}

\usepackage[left=3cm,right=3cm,top=3cm,bottom=3cm]{geometry}

\usepackage[hyperpageref]{backref}
\usepackage{mathrsfs}
\usepackage{lmodern}

\numberwithin{equation}{section}

\newtheorem{theorem}{Theorem}[section]
\newtheorem{lemma}[theorem]{Lemma}

\newtheorem{proposition}[theorem]{Proposition}

\theoremstyle{definition}

\newtheorem{remark}[theorem]{Remark}

\newcommand{\N}{{\mathbb N}}

\newcommand{\R}{{\mathbb R}}
\renewcommand{\S}{{\mathbb S}}
\newcommand{\eps}{\varepsilon}

\newcommand{\weakto}{\rightharpoonup}

\newcommand{\weakstarto}{\stackrel{*}{\rightharpoonup}}

\newcommand{\cal}{\mathcal}
\def\Xint#1{\mathchoice
{\XXint\displaystyle\textstyle{#1}}%
{\XXint\textstyle\scriptstyle{#1}}%
{\XXint\scriptstyle\scriptscriptstyle{#1}}%
{\XXint\scriptscriptstyle\scriptscriptstyle{#1}}%
\!\int}
\def\XXint#1#2#3{{\setbox0=\hbox{$#1{#2#3}{\int}$ }
\vcenter{\hbox{$#2#3$ }}\kern-.6\wd0}}

\def\dashint{\Xint-}

\title{Nonlocal problems at nearly critical growth}

\author[S.\ Mosconi]{Sunra Mosconi}
\address[S.\ Mosconi]{Dipartimento di Informatica
\newline\indent
Universit\`a degli Studi di Verona
\newline\indent
C\'a Vignal 2, Strada Le Grazie 15, 37134 Verona, Italy}
\email{sunra.mosconi@univr.it}

\author[M.\ Squassina]{Marco Squassina}
\address[M.\ Squassina]{Dipartimento di Informatica
\newline\indent
Universit\`a degli Studi di Verona
\newline\indent
C\'a Vignal 2, Strada Le Grazie 15, 37134 Verona, Italy}
\email{marco.squassina@univr.it}



\thanks{The authors were partially supported by Gruppo Nazionale per l'Analisi Matematica, la Probabilit\`a e le loro Applicazioni (INdAM)}

\begin{document}

\begin{abstract}
We study the asymptotic behavior of solutions to the nonlocal nonlinear 
equation $(-\Delta_p)^s u=|u|^{q-2}u$ in a bounded domain $\Omega\subset\R^N$ as $q$ approaches 
the critical Sobolev exponent $p^*=Np/(N-ps)$. We prove that ground state solutions concentrate at a single point 
$\bar x\in \overline\Omega$ and analyze the asymptotic behavior for sequences of solutions at higher energy levels. 
In the semi-linear case $p=2,$ we prove that for smooth domains the concentration point $\bar x$ cannot lie on the 
boundary, and identify its location in the case of annular domains.
\end{abstract}

\maketitle

%

\section{Introduction and main results}
Let  $\Omega$ be a smooth bounded domain of $\R^N$, $N>ps$, $s\in (0,1)$ and $p>1$. 
For any sufficiently small $\eps>0,$ we consider the nonlocal nonlinear problem
\begin{equation}
\begin{cases}
\label{probcrit}
(-\Delta_p)^s u=|u|^{p^*-2-\eps}u, & \text{in $\Omega$,} \\
\,\, u=0 & \text{in $\R^N\setminus\Omega$,}
\end{cases}
\end{equation}
where $p^*=Np/(N-sp)$ is the critical exponent for the immersion of
\[
W^{s,p}_0(\Omega):=\left\{u\in L^p(\R^N):
\int_{\R^{2N}}\frac{|u(x)-u(y)|^p}{|x-y|^{N+ps}}\, dx<+\infty,\,\,\, u=0 \text{ in $\R^N\setminus\Omega$}\right\}
\]
into the space $L^q(\Omega)$.\ By a weak solution to problem
\eqref{probcrit} we mean a critical point for the $C^1$ functional $I_\eps:W^{s,p}_0(\Omega)\to\R$ defined by
$$
I_\eps(u)=\frac{1}{p}\int_{\R^{2N}}\frac{|u(x)-u(y)|^p}{|x-y|^{N+sp}}\, dx\, dy-
\frac{1}{p^*-\eps}\int_{\Omega}|u|^{p^*-\eps}dx.
$$
The nonlinear operator $(-\Delta_p)^s : W^{s,p}_0(\Omega)\to W^{-s,p}(\Omega)$ is defined (up to a multiplicative constant 
which we will ignore in the following) as the differential of the first term in $I_\eps$ and it can be represented, on smooth functions, by
\begin{equation*}
	(-\Delta_p)^s \, u(x)= 2 \lim_{\varepsilon \searrow 0} \int_{\mathbb{R}^N \setminus B_\varepsilon(x)}
	\frac{|u(x) - u(y)|^{p-2}\, (u(x) - u(y))}{|x - y|^{N+s\,p}}\, dy, \qquad x \in \mathbb{R}^N.
\end{equation*}
Solutions of \eqref{probcrit} inherit some mild smoothness when seen as general non-homogeneous non-local equations. The regularity theory for $p\neq 2$ is far from complete, however the basic continuity instances of it are covered in \cite{KMS}, \cite{BPa} at the interior and in \cite{IMS} at the boundary.\\
In this paper, we are interested in the asymptotic behavior of a sequence of solutions $\{u_\eps\}_{\eps>0}$ 
to \eqref{probcrit} as $\eps\searrow 0$, as determined by the limit energy $c=\lim_{\eps\to 0} I_\eps(u_\eps)$. 

\vskip 2pt

The interest in such ``nearly critical" problems arises from the fact that for $\eps>0$ compactness is recovered and the problem is more easily solved,  hopefully providing in the limit a solution to the non-compact problem at $\eps=0$.\ In many cases, however, the validity of a Pohozaev identity rules out existence of nontrivial solutions for $\eps=0$, and the asymptotic behavior of the approximating solutions describes the phenomenon of lack of solutions in the limit.

\vskip 2pt

In the seminal paper \cite{AP}, the asymptotic behaviour of the (unique and radial) solution $u_\eps$
\begin{equation}
\label{locprob}
\begin{cases}
-\Delta u_\eps=u_\eps^{2^*-\eps},\quad u_\eps>0 &\text{ in $B$},\\
u_\eps=0 &\text{on $\partial B$}
\end{cases}
\quad \text{where $B$ is a ball in $\R^3$}
\end{equation}
is considered, showing, among other things, that $u_\eps$ concentrates at a single point, the center of $B$, at a rate $\max u_\eps=u_\eps(0)\simeq \sqrt{\eps}$. This kind of results were extended and refined in \cite{BP}.\  For general smooth domains, where uniqueness of solutions (and nonexistence of the latter for $\eps=0$) to \eqref{locprob} is lost, the same kind of behavior is proved in \cite{H}, \cite{R} for the {\em ground states} of \eqref{locprob}, namely, nontrivial solutions minimizing the associated energy functional. Indeed, regardless of the existence of positive solutions of the limiting equation, ground states {\em always} concentrate {\em all} their mass at some point, which is therefore called the point of concentration. Through a rather fine analysis, the concentration point is shown to be a critical point of the Robin function of $\Omega$. For smooth domains, this implies that the concentration points cannot belong to $\partial\Omega$, while for nonsmooth domain the boundary concentration phenomenon can happen, as shown e.g.\ in \cite{FGM}.
\vskip 2pt

For more general, nonlinear equations, the situation is less clear. In \cite{GAPA} the concentration of ground states is proved for the $p$-Laplacian via critical point methods, while in \cite{Pal1} via $\Gamma$-convergence ones (see the latter for more references on this approach). In \cite{MS} more general and non regular operators are considered. However, the location of concentration points for ground states is not clear, (even trying to prove that they do not belong to $\partial\Omega$ in smooth domains), and precise asymptotic behavior of the maxima are even less so. It is worth noting, however, that for a different but related problem involving the $p$-Laplacian, the location of  concentration points has been determined with the technique of $p$-harmonic transplantation, see \cite{flucher}.
\vskip2pt

Regarding the nonlocal problem \eqref{probcrit}, the semi-linear case $p=2$ is considered in \cite{PPS} with a $\Gamma$-convergence approach and in  \cite{PP} via profile decomposition. The latter approach relies on the Hilbert structure to take advantage of abstract profile decomposition theory, but, as shown in \cite{J}, no such precise decomposition can hold {\em for general bounded sequences} when $p\neq 2$. A more suitable profile decomposition when $p\neq 2$ has recently been obtained {\em for Palais-Smale sequences} in \cite{BSY}, which in principle may lead to the same kind of results we will discuss in a short while. However, a direct approach through non local Concentration-Compactness seems more convenient for ground state solutions, and is flexible enough to provide informations at higher critical levels as well. 

\vskip 2pt 
In order to state our main results, let us set
\[
|D^s u|^p(x):=\int_{\R^N} \frac{|u(x)-u(y)|^p}{|x-y|^{N+ps}}\, dy,\quad \text{for a.e. $x\in\R^N$}.
\]
For general $p>1$, $s\in \ ]0,1[$ and $N>ps$ we will prove the following.

\begin{theorem}[Ground states] 
\label{G-S} Let $\Omega$ be a bounded domain and $\{u_n\}$ be a sequence of ground state solutions to \eqref{probcrit} for $\eps_n\downarrow 0$. For each subsequence (not relabelled) such that $|u_n|^{p^*}$ weakly converges in the sense of measure, it holds
\begin{enumerate}
\item 
$u_n\to 0$ in $L^q(\R^N)$ for any $q\in [1,p^*)$;
\item
there exists $\bar x\in \overline\Omega$ such that $|D^s u_n|^p\rightharpoonup \S^{\frac{N}{ps}}\delta_{\bar x}$ and $|u_n|^{p^*}\rightharpoonup \S^{\frac{N}{ps}}\delta_{\bar x}$ weakly in the sense of measures.
\end{enumerate}
\end{theorem}

\begin{theorem}[Higher states]
\label{H-E}
Let $\Omega$, $u_n$, $\eps_n$ be as above and suppose that 
\begin{equation}
\label{Bsl}
I_{\eps_n}(u_n)\to c\in \left] \frac{s}{N}\S^{\frac{N}{ps}}, 2\frac{s}{N}\S^{\frac{N}{ps}}\right[,
\end{equation}
and let, up to subsequences, $u_n\weakto u$ in $W^{s,p}_0(\Omega)$.
 Then one of the following alternatives holds:
 \begin{enumerate}
 \item
  $u_n\to u$ strongly in $W^{s,p}_0(\Omega)$ with $u$ being a positive or negative solution of the limit equation \eqref{probcrit} for $\eps=0$.
  \item
  $u=0$ and there exists $\bar x\in \overline\Omega$ such that $|D^s u_n|^p\weakto c\frac{N}{s}\delta_{\bar x}$ weakly in the sense of measures.
  \end{enumerate}
  \end{theorem}
  
  In the case $p=2$ and $\Omega$ smooth we can exclude that the concentration points lie on the boundary, and in some cases precisely locate them.
  
   \begin{theorem}[Inner concentration]
 \label{G-S2}
 Let $N>2s$ and $\Omega$ be a bounded $C^{1,1}$ domain. For any $\{\eps_n\}$ with $\eps_n\downarrow 0$, let $u_n$ be a ground state solution of 
 \begin{equation}
 \label{gs2}
 \begin{cases}
(-\Delta)^s u_n=u_n^{2^*-\eps_n-1},&\text{$u>0$ in $\Omega$},\\
u_n=0&\text{in $\Omega^c$}.
\end{cases}
\end{equation}
Then, up to subsequence, $|u_{n}|^{2^*}\weakto \S^{\frac{N}{2s}} \delta_{\bar x}$ weakly in the sense of measures, for  some $\bar x\in \Omega$.
\end{theorem}

\begin{theorem}[Location for annuli]
\label{anello}
Let $r_2>r_1>0$, $\Omega=B_{r_2}\setminus B_{r_1}$ and $N>2s$. Then the ground state 
solutions of \eqref{gs2} concentrate at  points $\bar x$ with $|\bar x|$ the harmonic mean of $r_2$ and $r_1$.
\end{theorem}

\noindent
Let us point out the main features of the previous theorems. In order to prove Theorem \ref{G-S}, we will derive a Concentration-Compactness alternative for nonlocal problems. In order to do so we proceed directly, controlling through a domain decomposition and Lemma \ref{lemmavanish} below all the nonlocal interactions which, in principle, could contribute to the limiting measures. In a similar way we prove a bound from below on the singular part arising from concentration. One fundamental difficulty in the case $p\neq 2$ is that entire ground state solutions of the limiting problem are unknown explicitly. Thus,  we will have to use an auxiliary function recently constructed in \cite{PYMS} to prove that the energy of the ground states converges to the minimal one. To prove Theorem \ref{H-E} we will bound from below the absolutely continuous part of the limiting measures, again through the
crucial Lemma \ref{lemmavanish}. 
\vskip2pt
\noindent
For Theorems~\ref{G-S2} and \ref{anello} we will employ the Moving Plane method which has recently been proved in \cite{BMSc} in the nonlocal case, together with suitable fractional Kelvin transforms. Our main lemma here is a Harnack inequality on segments for positive solutions $u$ of a large class of semilinear equations. The inequality is of the form
\[
\sup_{[a, b]}u\leq Cu(b)
\]
where $[a, b]\subseteq \Omega$ is a segment of suitable fixed length, normal to $\partial\Omega$ at the point $a$, with $C$ being a geometric constant depending only on the domain. For a precise statement see Theorem~\ref{thmstima}. This forces the concentration to happen at least a distance $|b-a|$ away from the boundary. The construction can be performed explicitly for annular domains, yielding Theorem \ref{anello}.

\vskip6pt
\noindent
{\em Structure of the paper}.\
In Section 2.1 we fix some notations. 
In Section 2.2, after recalling some well known facts on the space $W^{s,p}_0(\Omega)$ and the ground states 
solutions of \eqref{probcrit}, we will prove some general results on the weak solutions of \eqref{probcrit} for fixed $\eps>0$. 
Section 2.3 is dedicated to the nonlocal Concentration-Compactness lemma Theorem~\ref{lions}. 
Then, in Section 3, we will prove Theorems \ref{G-S} and \ref{H-E}, and in Section 4 Theorem~\ref{G-S2} and \ref{anello}.

\section{Preliminaries}
\subsection{Notation}
For $s\in (0, 1)$, $p>1$ and $N> ps$, we let $p^*=Np/(N-ps)$.
We denote by $\omega_N$ the measure of the $N-$dimensional ball having unit radius.
For $E\subseteq \mathbb{R}^N$ measurable we denote by $|E|$ its $N-$dimensional Lebesgue measure, by $E^c=\mathbb{R}^N\setminus E$ its complement and by $\chi_E$ its characteristic function. If $u:E\to \mathbb{R}$ is measurable we set
\[
[u]_{W^{s, p}(E)}^p:=\int_{E\times E}\frac{|u(x)-u(y)|^p}{|x-y|^{N+ps}}\, dx\, dy, \qquad [u]_{s,p}:=[u]_{W^{s,p}(\mathbb{R}^N)}, 
\]
and for any $q\geq 1$
\[
|u|_{L^q(E)}:=\left(\int_{E}|u|^q\, dx\right)^{1/q},\qquad |u|_q:=|u|_{L^q(\mathbb{R}^N)}.
\]
Finally, for $t\in \mathbb{R}$ and $p>1$, we will use the notation
\[
t^{p-1}:=|t|^{p-2}t.
\]
For any $u\in W_0^{s,p}(\Omega)$  and any Lebesgue point $x$ for $u$, we set
\[
|D^s u|^p(x)=\int_{\R^N} \frac{|u(x)-u(y)|^p}{|x-y|^{N+ps}}\, dy.
\]
Notice that the following approximate Leibniz  formula holds true
\begin{equation}
\label{lieb}
\forall \theta>0 \ \exists C_\theta \text{ with}
\quad \int_{\R^N} |D^s (uv)|^p\, dx\leq(1+\theta)\int_{\R^N} |D^s u|^p|v|^p\, dx+C_\theta\int_{\R^N} |D^sv|^p|u|^p\, dx,
\end{equation}
which follows from the elementary inequality $|a+b|^p\leq (1+\theta)|a|^p+C_\theta|b|^p$.

\subsection{Functional Analytic framework}
We let, for any $\Omega\subseteq \R^N$,
\[
W^{s,p}_0(\Omega):=\Big\{u\in L^{p^*}(\R^N): \text{$u=0$ in $\Omega^c$},\,\, [u]_{s,p}<+\infty\Big\},
\]
which is a Banach space with respect to the norm $[\,\cdot\, ]_{s,p}$.\ This space is often denoted by $D^{s,p}_0(\Omega)$
in the current literature and it is consistent with the one defined in 
the introduction as soon as $\Omega$ is bounded. For $N>sp$, the Sobolev inequality reads as
\[
\S:=\inf\left\{\frac{[u]_{s,p}^p}{|u|_{p^*}^p}: u\in W^{s,p}_0(\Omega)\setminus\{0\}\right\}.
\]
We recall the following
\begin{proposition}[Hardy's inequality]
Let $N>sp$. Then there exists $C=C(N, p, s)>0$ such that 
\[
\int_{\mathbb{R}^N}\frac{|u|^p}{|x|^{sp}}\, dx\leq C[u]_{s, p}^p,\qquad \mbox{ for every } u\in W^{s,p}_0(\mathbb{R}^N).
\]
\end{proposition}

\noindent
We will also let, for $1\leq q\leq p^*$,
\begin{equation}
\label{infq}
\S_q(\Omega):=\inf\left\{\frac{[u]_{s,p}^p}{|u|_{q}^p}: u\in W^{s,p}_0(\Omega)\setminus\{0\}\right\},
\end{equation}
so that $\S=\S_{p^*}(\Omega)$.\ It is a classical fact that $\S_{p^*}(\Omega)=\S_{p^*}(\R^N)$ for every 
domain $\Omega$ and that the minimization problem \eqref{infq} for $q=p^*$ admits no solution, unless $\Omega=\R^N$. For $1< q<p^*$ and  $\Omega$ bounded, H\"older's inequality ensures that $\S_q(\Omega)>0$; moreover, any minimizer $u$  for the minimization problem \eqref{infq} (which, actually, exists due to the compactness of $W^{s,p}_0(\Omega)\hookrightarrow L^q(\Omega) $) is, up to a multiplicative constant, a weak solution of the problem
\begin{equation}
\label{probq}
\begin{cases}
(-\Delta_p)^s u=|u|^{q-2}u &\text{in $\Omega$},\\
u= 0 &\text{in $\Omega^c$}.
\end{cases}
\end{equation}
 Clearly any such solution is also a critical point for the functional 
\[
J_q(u):=\frac{1}{p}[u]_{s,p}^p-\frac{1}{q}|u|_q^q
\]
where  $J_q\in C^1(W^{s,p}_0(\Omega))$. If
\[
{\cal N}_q(\Omega):=\{ u\in W^{s,p}_0(\Omega)\setminus \{0\}:\langle dJ_q(u), u\rangle=0\}
\]
is the corresponding Nehari manifold for $J_q$, one readily checks that 
\begin{equation}
\label{gs}
c_q=\inf\big\{J_q(u):u\in {\cal N}_q(\Omega)\big\}=\left(\frac{1}{p}-\frac{1}{q}\right)\S_q(\Omega)^{\frac{q}{q-p}},
\end{equation}
and minimizers for the latter problem are (up to multiplicative constants) minimizers for \eqref{infq}. Any minimizer for \eqref{gs} is called a {\em ground state} solution for \eqref{probq}, being a critical point of minimal energy. Following \cite[Theorem 4.2]{Willem} we have also the following 
Mountain Pass characterization of the values $c_q$, valid for every $q>p$
\[
c_q=\inf_{\substack{u\in W^{s,p}_0(\Omega)\\ u\neq 0}}\max_{t\geq 0} J_q(tu)=\inf_{\gamma\in \Gamma}\max_{t\in [0,1]}J_q(\gamma(t)),
\]
where
\[
\Gamma:=\big\{\gamma\in C^0\big([0,1], W^{s,p}_0(\Omega)\big):\gamma(0)=0, J_q\big(\gamma(1)\big)<0\big\}.
\]
Any ground state solution must be of constant sign, since it also solves problem \eqref{infq}, and 
\[
[|u|]_{s,p}<[u]_{s,p}\qquad \text{if $u_+$, $u_-\neq 0$}.
\]
A more precise statement is given in the following proposition (see also \cite[Lemma 2.13]{BSY})
\begin{lemma}[Energy doubling]
\label{sign}
	Let $q\geq p$ and $u\in W^{s,p}_0(\Omega)$ be a sign-changing weak solution to
	\begin{equation}
		\label{2.10}
		\begin{cases}
			(-\Delta_p)^s  u  = |u|^{q-2}u & \text{in $\Omega$}  \\
			u=0 & \text{in $\R^N\setminus \Omega$},
		\end{cases}
	\end{equation}
	where $\Omega$ is an arbitrary open subset of $\R^N$.\ Then
	$$
	J_q(u)> 2\left(\frac{1}{p}-\frac{1}{q}\right)\S_q(\Omega)^{\frac{q}{q-p}}.
	$$
	\end{lemma}
\begin{proof}
	If $u_\pm:=\max\{\pm u,0\}\in W^{s,p}_0(\Omega)\setminus \{0\}$, then 
	\begin{equation*}
			\pm (u(x)-u(y))^{p-1}(u_\pm(x)-u_\pm(y))\geq |u_\pm(x)-u_\pm(y)|^p, 
	\end{equation*}
for a.e.\ $x, y\in \R^N$. Indeed, for the inequality involving the positive part, we have
	\[
	\begin{split}
	&(u(x)-u(y))^{p-1}(u_+(x)-u_+(y))=|u_+(x)-u_+(y)|^p\chi_{\{u(x)\geq 0, u(y)\geq 0\}}\\
	&+(u_+(x)+u_-(y))^{p-1}u_+(x)\chi_{\{u(x)>0, u(y)<0\}}+(u_-(x) +u_+(y))^{p-1}u_+(y)\chi_{\{u(x)<0, u(y)>0\}}\\
	&\geq |u_+(x)-u_+(y)|^p.
	\end{split}
	\]
	A similar justification holds for the inequality involving the negative part.\ This also shows that 
	the above inequalities are strict as long as $u_\pm\neq 0$.
	Then, testing problem \eqref{2.10} by  $\pm u_\pm$ yields
	\begin{equation*}
		[u_\pm]_{s,p}^p \leq \int_{\R^{2N}}\frac{\pm (u(x)-u(y))^{p-1}(u_\pm(x)-u_\pm(y))}{|x-y|^{N+sp}}\, dx\, dy=\int_{\Omega}u_\pm^q\, dx, 
	\end{equation*}
	with strict inequality if $u_\pm\neq 0$. Now, letting
	\[
	\lambda_\pm:=\left(\frac{[u_\pm]_{s,p}^p}{|u_\pm|_q^q}\right)^{\frac{1}{q-p}},
	\]
	it holds $\lambda_\pm u_\pm\in {\cal N}_q(\Omega)$ and $\lambda_\pm <1.$ On the other hand
	\[
	\begin{split}
	J_q(u)&=\left(\frac{1}{p}-\frac{1}{q}\right)|u|^{q}_{q}=\left(\frac{1}{p}-\frac{1}{q}\right)\int_{\Omega}(u_+^{q}+u_-^{q})dx=
	\frac{1}{\lambda^q_+}J_q(\lambda_+u_+)+\frac{1}{\lambda_-^q}J_q(\lambda_-u_-)\\	
	&> 2\left(\frac{1}{p}-\frac{1}{q}\right)\S_q(\Omega)^{\frac{q}{q-p}},
	\end{split}
	\]
	which completes the proof.
\end{proof}

 We will choose in the following the nonnegative ground states, which actually turn out to be strictly positive in $\Omega$ by the following result. 
 
\begin{lemma}[Strong Maximum Principle]
Let $u\in W^{s,p}_0(\Omega)$ satisfy 
\begin{equation}
\label{crit}
\begin{cases}
(-\Delta_p)^su\geq 0 &\text{weakly in $\Omega$},\\
u\geq 0&\text{in $\Omega^c$}.
\end{cases}
\end{equation}
Then $u$ has a lower semi-continuous representative in $\Omega$, which is either identically $0$ or positive.
\end{lemma}
\begin{proof}
By the comparison principle of \cite[Proposition 2.10]{IMS}, we get $u\geq 0$ a.e.\ in $\R^N$. Proceeding as in \cite[Theorem 2.4]{IMS0} we find that $u$ 
admits a l.s.c.\ representative, which to ease the notation we will identify with $u$. Then, the set $\{x\in \Omega: u(x)=0\}$ is closed in $\Omega$. 
By the weak Harnack inequality \cite[Theorem 5.2]{IMS}, it holds
\[
\inf_{B_{R/4}} u\geq \sigma \left(\dashint_{B_R\setminus B_{R/2}} u^{p-1}\, dx\right)^{\frac{1}{p-1}},\qquad \forall B_R\subseteq \Omega.
\]
which implies that $\{x\in \Omega: u(x)=0\}$ is open in $\Omega$. 
Suppose $u\neq 0$, and let $\Omega=\cup_{j\in {\cal J}}\Omega_j$ where $\Omega_j$ are the connected components of $\Omega$. It follows from the previous discussion that for each $j\in {\cal J}$, either $u$ is strictly positive everywhere in $\Omega_j$ or it vanishes identically. Since $u\neq 0$, there is a connected component, say,  $\Omega_1$ such that $u>0$ in $\Omega_1$. Suppose now by contradiction that there exists another connected component, say, $\Omega_2$, such that $u\equiv 0$ in $\Omega_2$, and let $\varphi\in C^\infty_c(\Omega_2)$, $\varphi\geq 0$, $\varphi\neq 0$. Testing \eqref{crit} with $\varphi$ we get
\[
0\leq \int_{\R^{2N}} \frac{(u(x)-u(y))^{p-1}(\varphi(x)-\varphi(y))}{|x-y|^{N+ps}}\, dx\, dy=-2\int_{\Omega_2}\varphi(x)\int_{\Omega_2^c}\frac{u^{p-1}(y)}{|x-y|^{N+ps}}\, dy\, dx<0, 
\]
since $u>0$ in $\Omega_1\subseteq \Omega_2^c$.
\end{proof}

\begin{remark}
A similar statement is provided in \cite[Theorem A.1]{BF} with a different proof. Notice that, contrary to the local case $s=1$, connectedness of $\Omega$ is not required. This is a typical feature of nonlocal problems, which was first outlined in \cite{BPa}.
\end{remark}

 \subsection{Concentration-Compactness}

We now prove a concentration-compactness lemma which was first stated without proof by P.L.\ Lions in \cite[Remark I.6]{Lions}.  We say that a sequence of functions $\{f_n\}_n\subseteq L^1(\R^N)$ converges tightly to a Borel regular measure $d\mu$ if 
\[
\forall \varphi \in C_b(\R^N),\qquad \int_{\R^N}\varphi f_n\, dx\to \int_{\R^N}\varphi\, d\mu,
\]
where $C_b(\R^N)$ is the Banach space of bounded continuous functions on $\R^N$. Notice that this convergence is stronger than the usual weak convergence of measures as linear functionals on the separable space $C_0(\R^N)$: indeed boundedness of $\{|f_n|_1\}$ does not suffice to the sequential compactness with respect to tight convergence. Nevertheless, we will still denote by with the symbol $\weakstarto$ the notion of tight convergence. Prokhorov theorem ensures that bounded sequences $\{f_n\}_n$ are relatively sequentially compact if and only if the sequence is {\em tight} in the sense that
\[
\forall \eps>0 \ \exists A\subseteq\R^N:\quad \sup_n\int_{A^c}|f_n|\, dx<\eps.
\]

\begin{theorem}
	\label{lions}
Let $\{u_n\}$ be a bounded sequence in $W^{s,p}_0(\Omega)$. Then, up to a subsequence, there exists $u\in W^{s,p}_0(\Omega)$, two Borel regular measures $\mu$ and $\nu$, $\Lambda$ denumerable, $x_j\in\overline{\Omega}$, $\nu_j\geq 0$, $\mu_j\geq 0$ with $\nu_j+\mu_j>0$ $j\in\Lambda$, such that 
\begin{align}
&u_n\to  u\quad \text{weakly in $W^{s,p}_0(\Omega)$ and strongly in $L^p(\Omega)$},\nonumber \\
\label{meas}
&|D^s u_n|^p\weakstarto d\mu,\quad |u_n|^{p^*}  \weakstarto d\nu  \\
 \label{mu}
 &d\mu \geq  |D^s u|^p+\sum_{j\in\Lambda} \mu_j \delta_{x_j}, \qquad \mu_j:=\mu(\{x_j\}),\\
 \label{nu}
 &d\nu = |u|^{p^*}+\sum_{j\in\Lambda} \nu_j \delta_{x_j}  , \qquad \nu_j:=\nu(\{x_j\}),\\
 \label{sobL}
 &\mu_j\geq \S\nu_j^{\frac{p}{p^*}}.
\end{align}
\end{theorem}

We will need the following lemma.

\begin{lemma}
\label{lemmavanish}
Let $N>ps$. For any $u\in L^{p^*}(\R^N)$ it holds
\[
\lim_{\delta\downarrow 0}\delta^N\int_{B_\delta^c}\frac{|u|^p}{|x|^{N+ps}}\, dx=0.
\]
\end{lemma}

\begin{proof}
If $u\in L^\infty(\R^N)$, the assertion immediately follows, since a direct computation yields
$$
\delta^N\int_{B_\delta^c}\frac{|u|^p}{|x|^{N+ps}}\, dx\leq C\|u\|_{L^\infty}^p\delta^{N-sp}.
$$	
In the general case, let $\{u_k\}\subset C^\infty_c(\R^N)$ be such that $u_k\to u$ in $L^{p^*}(\R^N)$ as $k\to\infty$. By using
H\"older inequality, for any $\delta>0$ and $k\in\N$, we obtain
\begin{align*}
\delta^N\int_{B_\delta^c}\frac{|u|^p}{|x|^{N+ps}}\, dx & \leq  C\delta^N \int_{B_\delta^c}\frac{|u_k-u|^p}{|x|^{N+ps}}\, dx+
C\delta^N\int_{B_\delta^c}\frac{|u_k|^p}{|x|^{N+ps}}\, dx \\
& \leq  C|u_k-u|^p_{p^*}+
C\delta^N\int_{B_\delta^c}\frac{|u_k|^p}{|x|^{N+ps}}\, dx.
\end{align*}
Since any $u_k$ is bounded, letting $\delta\to 0$ yields
\begin{align*}
\limsup_{\delta\to 0}\delta^N\int_{B_\delta^c}\frac{|u|^p}{|x|^{N+ps}}\, dx \leq  C|u_k-u|^p_{p^*}.
\end{align*}
Finally, letting $k\to\infty$ concludes the proof.
\end{proof}

Now we proceed proving Theorem \ref{lions}

\begin{proof}
Since $\Omega$ is bounded and $u_n\equiv 0$ in $\Omega^c$, the sequence $\{|u_n|^{p^*}\}$ is tight, ensuring the existence of $\nu$ (and clearly ${\rm supp}(\nu)\subseteq \overline\Omega$). To prove the tightness of $\{|D^su_n|^p\}$, let $U$ be open and  bounded such that $U\supset \overline\Omega$. If ${\rm dist}(U^c, \Omega)=: \theta>0$, then for any $x\in U^c$ and $y\in \Omega$ it holds $|x-y|\geq C_\theta|x|$, and thus
\[
|D^su_n|^p(x)=\int_{\Omega}\frac{|u_n(y)|^p}{|x-y|^{N+ps}}\, dy\leq C\frac{[u_n]_{s,p}^p}{|x|^{N+ps}}, \qquad \text{for a.e. $x\in U^c$}.
\]
The latter inequality readily implies tightness of $\{|D^su_n|\}$ and thus \eqref{meas} is proved.
 We come to the proof of \eqref{mu}, \eqref{nu} and \eqref{sobL}.\ We shall follow the proof of \cite[Lemma I.1]{Lions}, by supposing first that $u\equiv 0$.
From Sobolev's inequality and \eqref{lieb}, we have, for any $\varphi\in C^\infty_c(\R^N)$,
\begin{equation}
\label{temp}
\S|u_n\varphi|_{p^*}^{\frac{p}{p^*}}\leq [u_n\varphi]_{s,p}^p\leq (1+\theta)\int_{\R^N} |D^s u_n|^p|\varphi|^p\, dx+C_\theta\int_{\R^N} |D^s\varphi|^p|u_n|^p\, dx.
\end{equation}
Letting $n\to +\infty$ and using that 
$u_n\to 0$ in $L^p(\R^N)$ and $|D^s\varphi|^p\in L^\infty(\R^N)$, we obtain
\[
\S\left(\int_{\R^N} |\varphi|^{p^*}\, d\nu\right)^{\frac{p}{p^*}}\leq (1+\theta)\int_{\R^N} |\varphi|^p\, d\mu.
\]
Letting $\theta\downarrow 0$ proves \eqref{mu} and 
\eqref{nu}, due to \cite[Lemma I.2]{Lions}.\ Finally, the 
previous inequality easily implies, for any $j\in \Lambda$
\begin{equation}
\label{compara}
\S\big(\nu(B_\delta(x_j))\big)^{\frac{p}{p^*}}\leq (1+\theta)\mu(B_{2\delta}(x_j)),
\end{equation}
which provides \eqref{sobL} taking the limit for $\theta\downarrow 0$ and then $\delta\downarrow 0$.\ To prove the case $u\neq 0$, one can proceed as in \cite{Lions} to obtain \eqref{nu}. Concerning \eqref{mu}, we first claim that
$d\mu\geq |D^s u|^p$. Indeed, for any $\varphi\in C^\infty_c(\R^N)$, $\varphi\geq 0$, the functional 
\[
v\mapsto \int_{\R^N}|D^sv|^p\varphi\, dx
\]
is convex and continuous, therefore
$u_n\weakto u$ in $W^{s,p}_0(\Omega)$ implies
$$
\int_{\R^N}\varphi\, d\mu=\lim_{n\to\infty}\int_{\R^N} |D^s u_n|^p\varphi \, dx\geq \int_{\R^N} |D^s u|^p\varphi \, dx,\quad
\text{for any $\varphi\in C^\infty_c(\R^N)$,\,\, $\varphi\geq 0$}.
$$
On the other hand \eqref{compara} implies $\mu(\{x_j\})>0$ whenever $\nu(\{x_j\})>0$. Then \eqref{mu} follows.
We come to \eqref{sobL}. We take the limit for $n\to +\infty$ in \eqref{temp} to obtain
\[
\S\left(\int_{\R^N} |\varphi|^{p^*}\, d\nu\right)^{\frac{p}{p^*}}\leq (1+\theta)\int_{\R^N} |\varphi|^p\, d\mu+ 
C_\theta\int_{\R^N} |D^s\varphi|^p|u|^p\, dx.
\]
Now we fix $x_j$, and for any $\delta>0$, let $\varphi_\delta\in C^\infty_c(B_{2\delta}(x_j))$ satisfy
 \begin{equation}
 \label{defphi}
0\leq \varphi_\delta,\qquad \varphi\lfloor_{B_\delta}=1, \qquad | \varphi_\delta|_\infty\leq 1,\qquad |\nabla \varphi_\delta|_\infty\leq C/\delta.
\end{equation}
We claim the following:
\begin{equation}
\label{claim}
\forall u\in L^{p^*}(\R^N),\qquad \lim_{\delta\downarrow 0}\int_{\R^{2N}}|u|^p|D^s\varphi_\delta|^p\, dx= 0.
\end{equation}
Without loss of generality, suppose $x_j=0$. Let $A=(B_{2\delta}\times B_{2\delta})\cup (B_{\delta}\times B_{2\delta}^c)\cup (B_{2\delta}^c\times B_\delta)$ and notice that on $\R^{2N}\setminus A$ it holds $|\varphi_\delta(x)-\varphi_\delta(y)|=0$. Then on each of the three pieces forming $A$ we proceed as follows. Since $|\varphi_\delta(x)-\varphi_\delta(y)|\leq C\delta^{-1}|x-y|$, on $B_{2\delta}\times B_{2\delta}$ we have
\[
\begin{split}
\int_{B_{2\delta}\times B_{2\delta}}\frac{|u(y)|^p|\varphi_\delta(x)-\varphi_\delta(y)|^p}{|x-y|^{N+ps}}\, dx\, dy&\leq \int_{B_{2\delta}\times B_{4\delta}}|u(y)|^p\frac{C^p\delta^{-p}|z|^p}{|z|^{N+ps}}\, dy\, dz\\
&\leq \frac{C}{\delta^{ps}}\Big(\int_{B_{2\delta}}|u|^{p^*}\, dy\Big)^{\frac{p}{p^*}}|B_{2\delta}|^{1-\frac{p}{p^*}}=C\Big(\int_{B_{2\delta}}|u|^{p^*}\, dy\Big)^{\frac{p}{p^*}}
\end{split}
\]
which vanishes as $\delta\downarrow 0$.
On $B_{2\delta}^c\times B_{\delta}$, the triangle inequality implies that $|x-y|\geq |x|/2$ and thus, as $\delta\downarrow 0$,
\[
\begin{split}
\int_{B_{2\delta}^c\times B_{\delta}}\frac{|u(y)|^p|\varphi_\delta(x)-\varphi_\delta(y)|^p}{|x-y|^{N+ps}}\, dx\, dy&\leq C\int_{B_{2\delta}^c\times B_{\delta}}|u(y)|^p\frac{1}{|x|^{N+ps}}\, dx\,dy\\
&\leq C\Big(\int_{B_{\delta}}|u|^{p^*}\, dy\Big)^{\frac{p}{p^*}}|B_{\delta}|^{1-\frac{p}{p^*}}\frac{1}{\delta^{ps}}=o(1).
\end{split}
\]
Finally on $B_\delta\times B_{2\delta}^c$ it holds $|x-y|\geq |y|/2$ and thus
\[
\int_{B_{\delta}\times B_{2\delta}^c}\frac{|u(y)|^p|\varphi_\delta(x)-\varphi_\delta(y)|^p}{|x-y|^{N+ps}}\, dx\, dy\leq C\delta^N\int_{B_{2\delta}^c}\frac{|u(y)|^p}{|y|^{N+ps}}\, dy
\]
which vanishes as $\delta\downarrow 0$ by the previous lemma, and this proves \eqref{claim}.
Now if $\varphi_\delta=1$ in $B_\delta$ we obtain
\[
\S\big(\nu(B_\delta)\big)^{\frac{p}{p^*}}\leq (1+\theta)\mu(B_{2\delta})+C_\theta o(1)
\]
which gives \eqref{sobL} taking $\theta\downarrow 0$ and then $\delta\downarrow 0$. 

\end{proof}

\section{Limiting behaviour for $\eps\to 0$}

The next result provides lower bounds for the masses $\mu_j$ and $\nu_j$ given by the Concentration-Compactness theorem.

 \begin{lemma}
 \label{boundbelow}
 Let $u_\eps$ solve \eqref{probcrit}, and $\nu_j$, $\mu_j$ be as in \eqref{mu}-\eqref{nu}. Then, for any $j\in \Lambda$,
 \begin{equation}
 \label{masse}
  \mu_j \geq \S^{\frac{N}{sp}},\qquad  \nu_j\geq \S^{\frac{N}{sp}}.
  \end{equation}
  \end{lemma}

\begin{proof} 
Suppose again that $x_j=0$ and choose $\varphi_\delta\in C^\infty_c(B_{2\delta})$ as in formula \eqref{defphi}.
Testing the equation with $\varphi_\delta u_\eps$, we get
\begin{equation}
\label{eql}
\begin{split}
0&=\int_{\R^{2N}}\frac{(u_\eps(x)-u_\eps(y))^{p-1}(\varphi_\delta(x)u_\eps(x)-\varphi_\delta(y)u_\eps(y))}{|x-y|^{N+ps}}\, dx\, dy -\int_{\R^N}u_\eps^{p^*-\eps}\varphi_\delta\, dx\\
&\geq \int_{\R^{N}}|D^su_\eps|^{p}\varphi_\delta \, dx 
-\Big(\int_{\R^N}|u_\eps|^{p^*}\varphi_\delta\, dx\Big)^{1-\frac{\eps}{p^*}}|B_{2\delta}(x_j)|^{\frac{\eps}{p^*}}\\
&\quad  +\int_{\R^{2N}}\frac{(u_\eps(x)-u_\eps(y))^{p-1}u_\eps(y)(\varphi_\delta(x)-\varphi_\delta(y))}{|x-y|^{N+ps}}\, dx\, dy.
\end{split}
\end{equation}
Moreover, by H\"older's inequality, we get
\[
\left|\int_{\R^{2N}}\frac{(u_\eps(x)-u_\eps(y))^{p-1}u_\eps(y)(\varphi_\delta(x)-\varphi_\delta(y))}{|x-y|^{N+ps}}\, dx\, dy\right|\leq [u_\eps]_{s,p}^{p-1}\Big(\int_{\R^{N}}|D^s\varphi_\delta|^p|u_\eps|^p\, dy\Big)^{\frac 1 p}.
\]
Notice that $|D^s\varphi_\delta|^p\in L^\infty(\R^N)$, since 
\[
\int_{\R^N}\frac{|\varphi_\delta(x)-\varphi_\delta(y)|^p}{|x-y|^{N+ps}}\, dy
\leq \frac{C}{\delta^p}\int_{\R^N}\frac{\min\{1, |x-y|^p\}}{|x-y|^{N+ps}}\, dy\leq \frac{C}{\delta^p}.
\]
Up to subsequences we can suppose that $u_\eps\to u\in L^p(\Omega)$ as $\eps\to 0$ and, thus,
\[
\varlimsup_{\eps\to 0}\left|\int_{\R^{2N}}\frac{(u_\eps(x)-u_\eps(y))^{p-1}u_\eps(y)(\varphi_\delta(x)-\varphi_\delta(y))}{|x-y|^{N+ps}}\, dx\, dy\right|\leq C\Big(\int_{\R^{N}}|D^s\varphi_\delta|^p|u|^p\, dy\Big)^{\frac 1 p}.
\]
Taking the limit for $\eps\to 0$ in  \eqref{eql} we therefore obtained
\begin{equation}
\label{munu}
\int_{\R^N}\varphi_\delta\, d\mu\leq \int_{\R^N}\varphi_\delta\, d\nu+C\Big(\int_{\R^{N}}|D^s\varphi_\delta|^p|u|^p\, dy\Big)^{\frac 1 p}
\end{equation}
or, by \eqref{claim},
\[
\int_{\R^N}\varphi_\delta\, d\mu\leq \int_{\R^N}\varphi_\delta\, d\nu+o(1) ,
\]
with $o(1)\to 0$ as $\delta\downarrow 0$.\ This implies 
$\nu_j\geq \mu_j$, which, coupled together with \eqref{sobL}, gives \eqref{masse}.
\end{proof}

\begin{remark}
	From Lemma \ref{boundbelow}, the concentration points for $\nu$ agree with
	those for $\mu$ and, being $\nu$ of finite mass, the concentration set is finite, say ${\mathscr C}:=\{x_1, \dots, x_M\},$
	for some $M\in\N$. 
\end{remark}

\noindent
Let now $V$ be an optimizer for the Sobolev constant $\S$ which solves \eqref{probcrit} with $\eps=0$.
For $\delta>0$, we define the functions
\begin{equation*} 
V_\delta(x) := \frac{1}{\delta^{(N-sp)/p}}\; V\bigg(\frac{|x|}{\delta}\bigg),
\end{equation*}
and consider
\begin{equation*}
m_{\delta} := \frac{V_\delta(1)}{V_\delta(1) - V_\delta(\theta)}, \quad\,\,
G_{\delta}(t) := \begin{cases}
0, & 0 \leq t \leq V_\delta(\theta),\\
m_{\delta}\, (t - V_\delta(\theta)), & V_\delta(\theta) \leq t \leq V_\delta(1),\\
t, & t \geq V_\delta(1).
\end{cases}
\end{equation*}
where $\theta>1$ is a suitable constant introduced in \cite{PYMS} 
and the radially symmetric function
\[
v_{\delta}(r) = G_{\delta}(V_\delta(r))= 
\begin{cases}
V_\delta(r), & r \leq 1, \\[2pt]
0, & r \geq \theta.
\end{cases}
\]
We have the following estimates.

\begin{lemma} \label{immanuel}
	There exists a constant $C = C(N,p,s) > 0$ such that 
	\begin{align*}
	[v_{\delta}]_{s,p}^p &\leq \S^{\frac{N}{ps}} + C\delta^{\frac{N-ps}{p-1}}, \\
	\int_{\R^N} v_\delta^{p^*-\eps}dx &\geq \int_{\R^N} V_\delta(x)^{p^*-\eps}dx-C\delta^{(p^*-\eps)\frac{N-ps}{p(p-1)}}
	\end{align*}
for any $\delta \le 1/2$ and $\eps<N/(N-sp)$.
\end{lemma}
\begin{proof}
The first inequality was proved in \cite{PYMS}. Concerning the second inequality, we have
$$
\int_{\R^N} v_\delta^{p^*-\eps}dx\geq \int_{B_1} V_\delta(x)^{p^*-\eps}dx=
 \int_{\R^N} V_\delta(x)^{p^*-\eps}dx- \int_{B^c_1} V_\delta(x)^{p^*-\eps}dx.
$$
By virtue of \cite[Theorem 1.1]{BMS}, we have
\begin{align*}
\int_{B^c_1} V_\delta(x)^{p^*-\eps}dx&=\delta^{\frac{N-sp}{p}\eps}
\int_{B^c_{1/\delta}} V(x)^{p^*-\eps}dx \\
&\leq C\delta^{\frac{N-sp}{p}\eps}
\int_{B^c_{1/\delta}} \frac{1}{|y|^{\frac{N-sp}{p-1}(p^*-\eps)}}dy\leq 
C\delta^{\frac{(N-sp)(p^*-\eps)}{p(p-1)}},
\end{align*}
which concludes the proof.
\end{proof}

\begin{theorem}
\label{mp}
Let $\{u_n\}_n$ be a sequence of ground state solutions to \eqref{probcrit} for some $\eps_n\downarrow 0$. Then
$$
\lim_{n} I_{\eps_n}(u_n)=\frac{s}{N}\S^{N/ps}.
$$
\end{theorem}
\begin{proof}
By testing the equation by $u_n$ and using H\"older and Sobolev inequality we get
$$
[u_n]_{s,p}^p\geq \S^{\frac{p^*-\eps_n}{p^*-p-\eps_n}}|\Omega|^{-\frac{\eps_n p}{(p^*-p-\eps_n)p^*}}
$$	
Therefore, we have
$$
I_{\eps_n}(u_n)=\left(\frac{1}{p}-\frac{1}{p^*-\eps_n}\right)[u_n]_{s,p}^{p}
\geq \left(\frac{1}{p}-\frac{1}{p^*-\eps_n}\right)\S^{\frac{p^*-\eps_n}{p^*-p-\eps_n}}|\Omega|^{-\frac{\eps_n p}{(p^*-p-\eps_n)p^*}}, 
$$
yielding
$$
\lim_{n}I_{\eps_n}(u_n)\geq \frac{s}{N}\S^{\frac{N}{ps}}.
$$	
To prove the opposite inequality, suppose without loss of generality that $\bar B_\theta\subseteq \Omega$. Since $\lim_{t\to +\infty}I_{\eps_n}( tv_\delta)=-\infty$, by the definition of the mountain pass level it holds
\[
I_{\eps_n}(u_n)\leq \sup_{t>0}I_{\eps_n}(tv_\delta)=\Big(\frac{1}{p}-\frac{1}{p^*-\eps_n}\Big)\left(\frac{[ v_\delta]_{s,p}}{|v_\delta|_{p^*-\eps_n}}\right)^{p\frac{p^*-\eps_n}{p^*-p-\eps_n}}.
\]
Using the previous lemma we thus obtain
\[
\varlimsup_{n} I_{\eps_n}(u_n)\leq \frac{s}{N}\varlimsup_{n}\frac{\big(\S^{\frac{N}{ps}}+C\delta^{\frac{N-ps}{p-1}}\big)^{\frac{p^*-\eps_n}{p^*-p-\eps_n}}}{\big(|V_\delta|_{p^*-\eps_n}^{p^*-\eps_n}- C\delta^{(p^*-\eps_n)\frac{N-ps}{p(p-1)}}\big)^{\frac{p}{p^*-p-\eps_n}}}.
\]
Now by dominated convergence, we get 
\[
\lim_{n}|V_\delta|_{p^*-\eps_n}^{p^*-\eps_n}=|V_\delta|_{p^*}^{p^*}=\S^{\frac{N}{ps}}
\]
and therefore
\[
\varlimsup_{n} I_{\eps_n}(u_n)\leq \frac{s}{N}\frac{(\S^{\frac{N}{ps}}+C\delta^{\frac{N-ps}{p-1}}\big)^{\frac{p^*}{p^*-p}}}{(\S^{\frac{N}{ps}}- C\delta^{\frac{N}{p-1}})^{\frac{p}{p^*-p}}}\leq \frac{s}{N}\S^{\frac{N}{ps}}(1+C\delta^{\frac{N-ps}{p-1}})
\]
for some universal constant $C=C(N, p, s)$. Letting $\delta\to 0$ proves the claim.
	\end{proof}

\vskip 2pt
\noindent
{\em Proof of Theorem \ref{G-S}}. The proof follows from  \eqref{masse} and Theorem \ref{mp} exactly as in \cite{GAPA}, providing us with $\bar x\in \overline{\Omega}$ satisfying both conditions.  \qed

\vskip6pt

As a direct application we have the following multiplicity result on annular domains.

\begin{theorem}
Let $\Omega= B_{R}\setminus B_{r}$ for some $R>r>0$. Then for $q<p^*$ sufficiently near $p^*$ there is a {\em continuum} of positive solutions to 
\[
\begin{cases}
(-\Delta_p)^su=u^q&\text{in $\Omega$},\\
u=0 &\text{in $\Omega^c$}.
\end{cases}
\]
\end{theorem}

\begin{proof}
Suppose not. Then there is a sequence $\eps_n\downarrow 0$ such that there is only a denumerable set of positive solutions to \eqref{probcrit} for any $n$. 
Let $q=p^*-\eps_n$ and $\{u_n\}_n$ be a nonnegative ground state for \eqref{probcrit}, which by the Strong Maximum Principle is strictly positive in $\Omega$. 
Then, by virtue of Theorem \ref{G-S}, there exists a point $\bar x\in\overline{\Omega}$ such that, up to a subsequence, 
$|u_n|^{p^*}\weakstarto \S^{\frac{N}{ps}}\delta_{\bar x}$.
For sufficently large $n$ it thus holds
\[
\int_{\R^N}|u_n|^{p^*}\, dx< \frac{3}{2}\S^{\frac{N}{ps}},\qquad \int_{B_r(\bar x)}|u_n|^{p^*}\, dx>\frac{3}{4}\S^{\frac{N}{ps}}.
\]
For any such $\eps_n\to 0$, the solution $u_n$ eventually cannot be radial since otherwise also the integral on $B_{r}(-\bar x)$ (disjoint from $B_r(\bar x)$
as $|\bar x|\geq r$) would be greater than $3\S^{N/ps}/4$, thus yielding
\[
\frac{3}{2}\S^{\frac{N}{ps}}>\int_{\R^N}|u_n|^{p^*}\, dx\geq \int_{B_r(\bar x)}|u_n|^{p^*}\, dx+\int_{B_r(-\bar x)}|u_n|^{p^*}\, dx>\frac{3}{2}\S^{\frac{N}{ps}},
\]
for sufficiently large $n=n_0\in\N$. The map 
\[
R\in {\rm SO}_N\to W^{s,p}_0(\Omega)\ni u_{n_0}\circ R
\]
is therefore a continuous, non-constant map, all of whose image is made of positive  
solutions to \eqref{probcrit}. This gives the contradiction.
\end{proof}

\begin{lemma}
Let $\Omega$ be a bounded domain and $\{u_n\}_n$ be a sequence of ground state solutions to \eqref{probcrit} for $\eps_n\downarrow 0$ such that $I_{\eps_n}(u_n)$ is bounded and Theorem \ref{lions} holds, with $u$ being its weak limit. Then if $u\neq 0$,
\begin{equation}
\label{absc}
\int_{\R^N} |D^s u|^p\, dx\geq \S^{\frac{N}{ps}},\qquad \int_{\R^N} |u|^{p^*}\, dx\geq \S^{\frac{N}{ps}}.
\end{equation}
\end{lemma}

\begin{proof}
Let ${\mathscr C}=\{x_1,\dots,x_M\}$ be the concentration set, and let $\varphi_\delta$ be the cut-off functions as introduced in \eqref{defphi}. Define, for any $\delta>0$ small enough, the function
\[
\psi_\delta(x):=\sum_{i=1}^M\varphi_\delta(x-x_i).
\]
Proceeding as in Lemma \ref{boundbelow}, testing \eqref{probcrit} with $(1-\psi_\delta)u_n$ and letting 
$n\to\infty$, we obtain
\begin{equation}
\label{comparat}
0\geq \int_{\R^N}(1-\psi_\delta)\, d\mu-\int_{\R^N}(1-\psi_\delta)\, d\nu-o(1),\quad\text{with $o(1)\to 0$ as $\delta \downarrow 0$.}
\end{equation}
Since $\psi_\delta=1$ on $\bigcup_{i=1}^MB_\delta(x_i)$ and $\psi_\delta=0$ on $\bigcap_{i=1}^MB_{2\delta}^c(x_i)$, we get
$$
\lim_{\delta \downarrow 0}\int_{\R^N}(1-\psi_\delta)\, d\nu=\nu(\R^N\setminus {\mathscr C}),\qquad
\lim_{\delta \downarrow 0}\int_{\R^N}(1-\psi_\delta)\, d\mu=\mu(\R^N\setminus {\mathscr C}),
$$
yielding in turn, by formulas \eqref{comparat} and \eqref{mu}, 
\[
\int_{\R^N} |u|^{p^*}\, dx=\nu(\R^N\setminus {\mathscr C})
\geq \mu(\R^N\setminus {\mathscr C})\geq  \int_{\R^N} |D^s u|^p\, dx.
\]
Using also Sobolev's inequality we deduce \eqref{absc} as long as $u\neq 0$.
\end{proof}

\vskip 2pt
\noindent
{\em Proof of Theorem \ref{H-E}.} By Lemma \ref{sign} we can suppose that $u_n\geq 0$, for sufficiently small $\eps_n>0$. Let $\mu$, $\nu$ be given in Theorem \ref{lions} and let ${\mathscr C}=\{x_1,\dots,x_M\}$ be their concentration set.
If $u\neq 0$, by the previous lemma and \eqref{masse}, we have
 \begin{equation}
 \label{sing}
 c=\lim_{\eps\to 0}\left(\frac{1}{p}-\frac{1}{p^*-\eps}\right)[u_n]_{s, p}^p=\frac{s}{N}\mu(\R^N)\geq\frac{s}{N}\int_{\R^N}|D^su|^p\, dx+\frac{s}{N} \sum_{i=1}^M\mu_j\geq \frac{s}{N}(M+1)\S^{\frac{N}{ps}}.
  \end{equation}
 Therefore $M=0$ by assumption \eqref{Bsl}. To prove that $u$ weakly solves 
 \begin{equation}
 \label{crit2}
 \begin{cases}
 (-\Delta_p)^s u=u^{p^*-1}&\text{in $\Omega$},\\
 u=0&\text{in $\Omega^c$},
 \end{cases}
 \end{equation}
  it suffices to show that, up to subsequences, $f_n:=u_n^{p^*-1-\eps_n}$ weakly converges to $f:=u^{p^*-1}$ in $L^{{p^*}'}(\R^N)$.  By compactness in $L^p(\Omega)$ we can  assume that $u_n\to u$ pointwise a.e. in $\R^N$ so that $f_n\to f$ pointwise a.e. in $\R^N$. Moreover 
\[
\begin{split}
\int_{\R^N}|f_n|^{\frac{p^*}{p^*-1}}\, dx&=\int_\Omega u_n^{p^*-\eps_n\frac{p^*}{p^*-1}}\, dx\\
&\leq \left(\int_\Omega |u_n|^{p^*}\, dx\right)^{1-\frac{\eps_n}{p^*-1}}|\Omega|^{\frac{\eps_n}{p^*-1}},
\end{split}
\]
so that $f_n \rightharpoonup f$ 
in $L^{{p^*}'}(\R^N)$ as $n\to\infty$. It follows that $u\in W^{s, p}_0(\Omega)$ weakly solves 
\eqref{crit2} in $\Omega$. We finally prove that $u_n\to u$ strongly in $W^{s,p}_0(\Omega)$. 
By H\"older's inequality it holds
\[
\int_{\R^N}u_n^{p^*}\, dx\geq \left(\int_{\R^N}u_n^{p^*-\eps_n}\, dx\right)^{\frac{p^*}{p^*-\eps_n}}|\Omega|^{-\frac{\eps_n}{p^*-\eps_n}},
\]
and thus, up to subsequences
\[
[u]_{s, p}^p=|u|_{p^*}^{p^*}=\nu(\R^N)=\varlimsup_{n} |u_n|_{p^*}^{p^*}\geq \varlimsup_{n} |u_n|_{p^*-\eps_n}^{p^*}|\Omega|^{-\frac{\eps_n}{p^*-\eps_n}}=\varlimsup_{n}[u_n]_{s,p}^{p\frac{p^*}{p^*-\eps_n}}= \varlimsup_{n}[u_n]_{s,p}^p.
\]  
Being $W^{s,p}_0(\Omega)$ uniformly convex, the claim follows.\
Suppose on the other hand that $u=0$.\ From formula \eqref{nu} we get $d\nu=\sum_{j=1}^M\nu_j\delta_{x_j}$ and we know that $d\mu\geq \sum_{j=1}^M\mu_j\delta_{x_j}.$ On the other hand
formula \eqref{munu}, which holds for arbitrary $\varphi\in C^\infty_c(\R^N)$ with $\varphi\geq 0$, implies $d\nu\geq d\mu$.\  Therefore
${\rm supp}(\mu)\subseteq {\rm supp}(\nu)={\mathscr C}$, which yields $d\mu=\sum_{j=1}^M\mu_j\delta_{x_j}$.\
Thus \eqref{sing} forces $M=1$ and $c=\frac{s}{N}\mu_1$, where $\mu_1>0$.
\qed

  \section{The case $p=2$}

Given $\Omega\subseteq \R^N$, we let in this section $\delta(x)={\rm dist}(x, \Omega^c)$. We will prove the following

\begin{theorem}
\label{thmstima}
Let $\Omega$ be a bounded $C^{1,1}$ domain. There exists a constant $d=d(\Omega)>0$ such that any weak solution of 
\begin{equation}
\label{prob2}
\begin{cases}
(-\Delta)^s u=f(u)&\text{in $\Omega$},\\
u=0&\text{in $\Omega^c$},
\end{cases}
\end{equation}
with $f\in {\rm Lip}_{\rm loc}(\R)$ such that
\begin{equation}
\label{subcrit}
 0\leq f(t)\leq C(1+|t|^{p^*}),\qquad r\mapsto \frac{f(r^{N-2s} t)}{r^{N+2s}} \quad \text{is non-increasing for any $t\geq 0$}
 \end{equation}
 satisfies for any $x_0\in \partial\Omega$, with exterior normal $\nu$
 \begin{equation}
 \label{stima}
 \xi_1<\xi_2\leq 2d\quad \Rightarrow \quad u(x_0-\xi_1\nu)< |1+2d|^{N-2s} u(x_0-\xi_2\nu).
 \end{equation}
 \end{theorem}
 
 \begin{remark}
 Notice that, if $u\geq 0$ a.e.\ in $\Omega$, then by the strong maximum principle, $u>0$ in $\Omega$ or $u\equiv 0$. Regarding the hypothesis \eqref{subcrit}, consider the homogeneous case $f(t)=t^{q-1}$, $N>2s$. Then the monotonicity property required in \eqref{subcrit} is satisfied if and only if $q\leq 2^*$. Finally, it is worth noting that the proof actually only requires that $\partial\Omega$ satisfy a uniform exterior sphere condition, so that \eqref{stima} actually holds, e.g., in convex domains.
 \end{remark}
 
 The main point of the theorem is that  the constant $d$ is {\em geometric} and independent on the nonlinearity, 
 as long as \eqref{subcrit} holds. Let us now show how \eqref{stima} implies that the ground states $u_\eps$ cannot concentrate on $\partial\Omega$.
To the best of our knowledge this argument is new.

\vskip 2pt
\noindent
{\em Proof of Theorem \ref{G-S2}}.\ Let $d>0$ be given in the previous theorem, so small that 
$$
\Pi(x):={\rm Argmin}\{|x-y|:y\in \partial\Omega\},\quad x\in N_t,
$$ 
is well defined and $C^{1,1}$, where $N_t:=\{x\in \overline\Omega:\delta(x)\leq t\}$ for any $0\leq t\leq 2d$. Then, denoting by $\nu_z$ the exterior normal to $\partial\Omega$ at $z$, the map  
\[
(z, \xi)\in \partial\Omega\times [0, 2d]\to \Phi(z, \xi)=z-\xi\nu_z\in N_{2d},
\]
is a bi-Lipschitz homeomorphism, with 
\begin{equation}
\label{lipphi}
\|{\rm det} D\Phi\|_{\infty}+\|{\rm det} D\Phi^{-1}\|_{\infty}\leq L=L(\Omega).
\end{equation}
 Suppose by contradiction that $\bar x\in \partial\Omega$, where $\bar x$ is the concentration point
 given by Theorem \ref{G-S}. Then for any small $\theta>0$ to be defined later, there exists a sufficiently large $n\in\N$ such that 
\[
\int_{N_{d}}|u_n|^{2^*}\, dx\geq \S^{\frac{N}{2s}}/2, \qquad \int_{N_{2d}\setminus N_d}|u_n|^{2^*}\, dx<\theta.
\]
Using the change of variables given by $\Phi$, the bound \eqref{lipphi} and \eqref{stima}, we have
\[
\begin{split}
\int_{N_{d}}|u_n|^{2^*}\, dx&\leq L\int_{\partial\Omega\times [0, d]}|u_{n}|^{2^*}(z-\xi\nu_z)\, dz\, d\xi\\
&\leq |1+2d|^{2N}L\int_{\partial\Omega\times [d, 2d]}|u_{n}|^{2^*}(z-\xi\nu_z)\, dz\, d\xi\\
&\leq  |1+2d|^{2N}L^2\int_{N_{2d}\setminus N_d}|u_n|^{2^*}\, dx.
\end{split}
\]
So that, choosing $\theta |1+2d|^{2N}L^2<\S^{\frac{N}{2s}}/2$, we get a contradiction.
\qed

\vskip 6pt
\noindent
{\em Proof of Theorem \ref{thmstima}}. The  proof relies on the moving plane method, performed after suitable Kelvin transforms through externally tangent balls.
 First of all, being $\Omega$ a $C^{1,1}$ bounded domain, there exists $r>0$ such that at any point $x_0\in \partial\Omega$ there exists an open  ball $B_r$ of radius $r$ with
\[
 B_r\subseteq \Omega^c,\qquad \partial B_r\cap\partial\Omega=\{x_0\}.
\]
 We can scale problem \eqref{prob2} in $r^{-1}\Omega$ and suppose that $r=1$. Henceforth, we will denote by $B(z)$ a ball of radius $1$ and center $z$.
 Given a function $u\in W^{s,p}_0(\Omega)$ and a ball $B(z)\subseteq \Omega^c$ such that $\partial B(z)\cap \partial\Omega=\{x_0\}$, the Kelvin transform of $u$ through $B(z)$ is given by 
 \[
 u^*(x)=\frac{1}{|x-z|^{N-2s}}u\Big(z+\frac{x-z}{|x-z|^2}\Big).
 \]
 The map ${\cal I}_z(x):=z+(x-z)/|x-z|^2$ brings $\Omega$ to $\Omega_z:={\cal I}_z(\Omega)\subseteq B(z)$. Furthermore,
 the function $u^*\in W^{s,p}_0(\Omega_z)$ satisfies
 \[
( -\Delta)^s u^*(x)=\frac{1}{|x-z|^{N+2s}}(-\Delta)^s u({\cal I}_z(x)).
\]
Given an externally tangent ball $B(z)$, the Kelvin transform thus brings any solution $u$ of \eqref{prob2} to a solution $v:=u^*$ in $\Omega_z\subseteq B(z)$ of 
\[
\begin{cases}
(-\Delta)^sv=g(x, v)&\text{in $\Omega_z$},\\
v=0&\text{in $\Omega_z^c$},
\end{cases}
\]
where $z\notin \overline{\Omega_z}$ and the nonlinearity
\[
g(x, t):=g(|x-z|, t)=\frac{f(|x-z|^{N-2s} t)}{|x-z|^{N+2s}}
\]
 satisfies the following properties:
\begin{enumerate}
\item
$g$ is Caratheodory, locally Lipschitz in the second variable and $C^1$ in the first;
\item
$g(r, t)$ is non-increasing in the first variable.
\end{enumerate}
From any $x_0\in \partial\Omega$, we will now construct the exterior tangent ball $B(z)$, $z=z(x_0)$ as above, and apply the moving plane method to $\Omega_z$.
Without loss of generality we now suppose $z=0$, calling $\Omega_0=\widetilde\Omega$. Let $\nu_{x_0}$ be the interior normal to $B(z)$ at $x_0=\partial\widetilde\Omega\cap \partial B$ and observe that 
\[
\inf_{\widetilde\Omega} x\cdot \nu_{x_0}=x_0\cdot \nu=-1,
\]
 due to our normalization. For any $\lambda>-1$ we let
\[
\widetilde{\Omega}^\nu_\lambda=\widetilde\Omega\cap \{x\cdot \nu<\lambda\}, \qquad x^\nu_\lambda=R^\nu_{\lambda}(x):=x+2(\lambda-x\cdot \nu)\nu,
\]
the latter being the  reflection of $x$ through the hyperplane $\{x\cdot\nu= \lambda\}$. 
By the regularity of $\partial\widetilde\Omega$, there is $\eps>0$ such that for all  $\lambda\in ]-1,-1+\eps]$, $R^\nu_\lambda(\widetilde\Omega^\nu_\lambda)\subseteq \widetilde\Omega$ (cf.\ \cite[Theorem 5.7]{F}). 
Moreover, by compactness, this $\eps>0$ can be chosen independently of the particular 
$x_0$ from which the construction started.\ Moreover, if $\eps<1$, it holds
\[
|x^\nu_\lambda|\leq |x|\quad \text{ for any $\lambda\in ]-1, -1+\eps]$, $x\in \widetilde\Omega^\nu_\lambda$}.
\]
Therefore, being $r\mapsto g(r, t)$ non-increasing for any $t\geq 0$, 
\[
g(x, t)\leq g(x^\nu_\lambda, t),\quad \text{ for any $\lambda\in ]-1, -1+\eps]$, $x\in \widetilde\Omega^\nu_\lambda$, $t\geq 0$}.
\]
The moving plane method as per \cite[Proposition 4.4]{BMSc} can be applied, giving that
\[
s\mapsto v(x_0+s\nu) \quad \text{is increasing for $s\in [0, \eps]$}.
\]
Recalling the definition of $v$ and $\nu=z-x_0$, we obtain
\[
s\mapsto \frac{1}{|1-s|^{N-2s}}u\Big(x_0-\nu\frac{s}{1-s}\Big)\quad \text{is increasing for $s\in [0, \eps]$}, 
\]
or, for $\xi=s/(1-s)$, 
\[
\xi\mapsto |1+\xi|^{N-2s}u(x_0-\xi \nu)\quad \text{is increasing for $\xi\in [0, \frac{\eps}{1-\eps}]$}.
\]
Since $\Omega$ is $C^{1,1}$, there is $\bar d\in ]0, 1[$ such that $\Pi(x)={\rm Argmin}\{|x-y|:y\in \partial\Omega\}$ 
is well defined and $C^{1,1}$. We define $d$ with $2d<\min\{\bar d, \eps\}<\eps/(1-\eps)$ so that the previous monotonicity gives
\[
\xi_1\leq d<\xi_2\leq 2d\quad \Rightarrow \quad u(x_0-\xi_1\nu)\leq \frac{|1+\xi_2|^{N-2s}}{|1+\xi_1|^{N-2s}}u(x_0-\xi_2\nu)\leq |1+2d|^{N-2s} u(x_0-\xi_2\nu),
\]
concluding the proof. \qed

\vskip 6pt
\noindent
{\em Proof of Theorem \ref{anello}}. Rescaling the problem if necessary we can reduce to 
$\Omega=B_{r_2/r_1}\setminus B_1$. Let $R=r_2/r_1$ and let, for $\eps>0$ sufficiently small, 
\[
\bar u_\eps=r_1^{\frac{2s}{2^*-2-\eps}}u_\eps(r_1x)
\]
 be the nonnegative solution to 
\[
\begin{cases}
(-\Delta)^s \bar u_\eps= \bar u_\eps^{2^*-1-\eps}&\text{in $\Omega$},\\
\bar u_\eps=0&\text{in $\Omega^c$}.
\end{cases}
\]
Using the minimizing property \eqref{gs} and scaling, it is readily checked that $\bar u_\eps$ is a ground state solution of the previous problem, for any $\eps>0$.
We proceed by contradiction and suppose that $|\bar u_\eps|^{2^*}\weakto \S^{\frac{N}{2s}}\delta_{x_0}$ for some $1<|x_0|<2/(1+R^{-1})$, letting
$M=2/(1+R^{-1}).$ 
We omit for the time being the dependance on $\eps$ and call $\bar u_\eps=u$. The ball $B_1$ is externally tangent to $\Omega$ and we can perform the Kelvin transform through $B_1$, obtaining a solution to 
\[
\begin{cases}
(-\Delta)^sv=\frac{v^{2^*-1-\eps}}{|x|^{\eps(N-2s)}}&\text{in ${\cal I}\Omega$},\\
v=0&\text{in ${\cal I}\Omega^c$}.
\end{cases}
\]
Notice that ${\cal I}\Omega=B_1\setminus B_{1/R}=:\widetilde{\Omega}$. 
  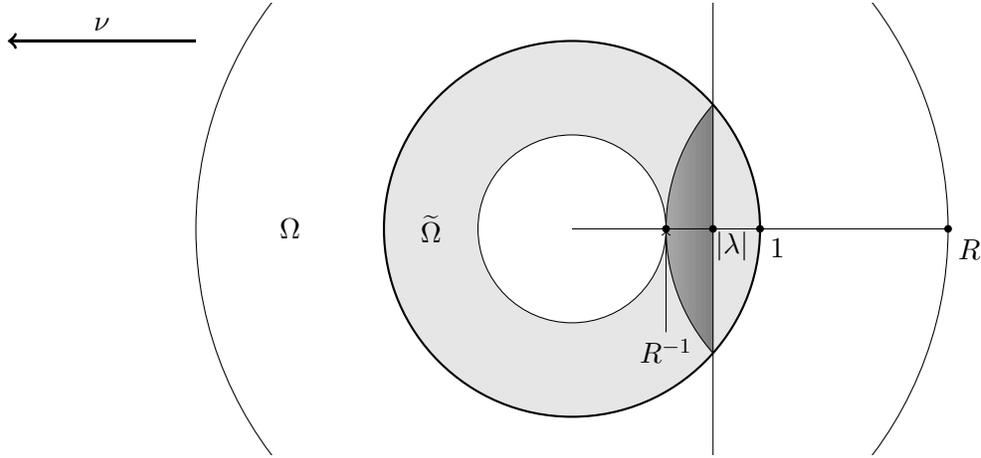
\begin{figure}[!tbp]
\centering
\begin{tikzpicture}[scale=2.5]
\clip (-3, -1.2) rectangle (3, 1.2);
\begin{scope}

\filldraw[gray!20] (0,0) circle (1cm);
\filldraw[white] (0,0) circle (0.5cm);

\end{scope}
\begin{scope}
\clip (1.5,0) circle (1cm);
\shade [left color=gray, right color=white, shading=axis, shading angle=270] (0, -2) -- (0.75, -2) -- (0.75, 2) -- (0, 2); 
\end{scope}
\draw (0,0) circle (0.5cm);
\draw[->] (0.5, -0.55) -- (0.5, -0.01);
\draw (0.5, -0.65) node{$R^{-1}$};
\filldraw (0.5,0) circle (0.5pt);
\draw[thick] (0,0) circle (1cm);
\draw (1, 0) node[below right]{$1$};
\filldraw (1,0) circle (0.5pt);
\draw (0,0) circle (2cm);
\draw (2,0) node[below right]{$R$};
\filldraw (2,0) circle (0.5pt);
\draw (0,0)--(2, 0);
\draw (-1.5,0) node{$\Omega$};
\draw (-0.75,0) node{$\widetilde{\Omega}$};
\begin{scope}
\clip (0,0) circle (1cm);
\draw (1.5,0) circle (1cm);
\end{scope}
\draw[thin] (0.75, -2) -- (0.75, 2); 
\draw (0.85, -0.09) node{$|\lambda|$};
\filldraw (0.75,0) circle (0.5pt);
\draw[very thick, ->] (-2, 1) -- (-3, 1); 
\draw (-2.5, 1) node[above]{$\nu$}; 
\end{tikzpicture}
\caption{The optimal cap after inversion through the inner sphere}
\label{inv1}
\end{figure}
Elementary geometric considerations (see figure \ref{inv1}) show that given any unit vector $\nu$, $a(\nu)=-1$ and  
\[
\frac{1}{R}+2(1-|\lambda|)<1\quad \Rightarrow \quad R^\nu_\lambda(\widetilde{\Omega}^\nu_\lambda)\subseteq \widetilde{\Omega}.
  \] 
 The previous condition actually gives the so called "optimal cap" (which in this case coincides with the maximal one), i.e.
 \[
 R^\nu_\lambda(\widetilde{\Omega}^\nu_\lambda)\subseteq \widetilde{\Omega}\quad \Leftrightarrow\quad -1<\lambda<-\frac{1+R^{-1}}{2}=-M^{-1},
\]
and the moving planes method ensures that 
\[
s\mapsto v((s-1)\nu) \quad \text{is increasing for $0\leq s\leq 1-M^{-1}$}.
\]
In terms of $u$ we therefore have, for any $x_0$ such that $|x_0|=1$ and for  $\nu=-x_0$, $t=1/(1-s)$,
\[
t\mapsto t^{N-2s} u(tx_0)\quad \text{is increasing for $1\leq t\leq M$}.
\]
Now if $|\bar x|<M-2\theta$ with $\theta>0$, the previous monotonicty ensures that 
\[
u_\eps(tx_0)<\frac{r^{N-2s}}{t^{N-2s}}u_\eps(rx_0)\leq M^{N-2s}u_\eps(rx_0),\quad \forall 1\leq t\leq M-\theta\leq r\leq M,
\]
and thus 
\[
\S^{\frac{N}{2s}}=\lim_{\eps\downarrow 0}\int_{B_{M-\theta}\setminus B_1} u_\eps^{2^*}\, dx\leq C_\theta M^{N+2s}\lim_{\eps\downarrow 0}\int_{B_M\setminus B_{M-\theta}}u_\eps^{2^*}\, dx=0,
\]
giving a contradiction. Similarly one can proceed if $|\bar x|>M+2\theta$: using this time Kelvin transform through an exterior unit ball of radius one (see figure \ref{inv2}), we obtain for any unit vector $x_0$ 
\[
t\mapsto (t+1)^{N-2s}u((R-t)x_0) \quad \text{is increasing for $0\leq t\leq M$},
\]
and integrating $u_\eps^{2^*}$ over $B_R\setminus B_{M+\theta}$ and $B_{M+\theta}\setminus B_M$ gives the desired contradiction.\qed

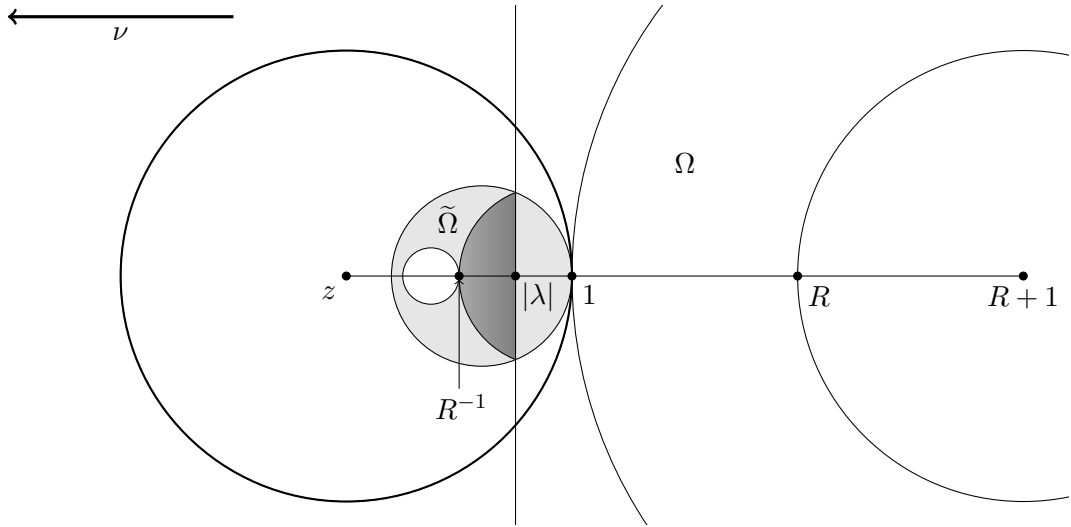
\begin{figure}[!tbp]
\centering
\begin{tikzpicture}[scale=3]
\clip (-1.5, -1.1) rectangle (3.2, 1.2);
\draw (0,0) circle (1cm);
\filldraw (0,0) circle (0.5pt);
\draw (0,0) node[below left]{$z$};

\begin{scope}
\filldraw[gray!20] (0.6,0) circle (0.4cm);
\filldraw[white] (0.375, 0) circle (0.125cm);
\end{scope}
\draw (0.6,0) circle (0.4cm);
\draw (0.375, 0) circle (0.125cm);

\begin{scope}
\clip (0.9,0) circle (0.4cm);
\shade [left color=gray, right color=white, shading=axis, shading angle=270] (0, -2) -- (0.75, -2) -- (0.75, 2) -- (0, 2); 
\end{scope}

\draw[->] (0.5, -0.5) -- (0.5, -0.015);
\draw (0.51, -0.58) node{$R^{-1}$};
\filldraw (0.5,0) circle (0.5pt);
\draw[thick] (0,0) circle (1cm);
\draw (1, 0) node[below right]{$1$};
\filldraw (1,0) circle (0.5pt);
\draw (3,0) circle (2cm);
\draw (3,0) circle (1cm);
\draw (2,0) node[below right]{$R$};
\filldraw (2,0) circle (0.5pt);
\draw (3,0) node[below ]{$R+1$};
\filldraw (3,0) circle (0.5pt);
\draw (0,0)--(3, 0);
\draw (1.5,0.5) node{$\Omega$};
\draw (0.45,0.25) node{$\widetilde{\Omega}$};
\begin{scope}
\clip (0.6,0) circle (0.4cm);
\draw (0.9,0) circle (0.4cm);
\end{scope}
\draw[thin] (0.75, -2) -- (0.75, 2); 
\draw (0.85, -0.09) node{$|\lambda|$};
\filldraw (0.75,0) circle (0.5pt);
\draw[very thick, ->] (-0.5, 1.15) -- (-1.5, 1.15); 
\draw (-1, 1) node[above]{$\nu$}; 
\end{tikzpicture}
\caption{The optimal cap after inversion through the outer sphere}
\label{inv2}
\end{figure}

\medskip

\end{document}